\documentclass[11pt]{article}
\usepackage{graphicx,amssymb,amsmath}
\usepackage{epsfig}
\usepackage{epsf}
\usepackage{amsfonts}
\usepackage{latexsym}
\usepackage{bbm}

\newtheorem{theorem}{Theorem}

\newtheorem{lemma}{Lemma}
\newtheorem{corollary}{Corollary}

\def \endproof {\hfill $\Box$ \smallskip}

\newcommand{\later}[1]{{}}
\newcommand{\old}[1]{{}}
\long\def\ignore#1{}

\newcommand{\NN}{\mathbb{N}} 
\newcommand{\RR}{\mathbb{R}} 

\providecommand{\intd}[0]%
{\;\mbox{d}}

\newcommand{\eps}{\varepsilon}
\newcommand{\dd}{\mathrm{d}}

\title{New bounds on the average distance from the\\
Fermat-Weber center of a planar convex body}

\author{%
Adrian Dumitrescu\footnote{Department of Computer Science,
University of Wisconsin--Milwaukee,
WI 53201-0784, USA\@. Email: \texttt{ad@cs.uwm.edu}.
Supported in part by NSF CAREER grant CCF-0444188.
}
\and
Minghui Jiang\footnote{Department of Computer Science, 
Utah State University, Logan, UT 84322-4205, USA\@.
Email: \texttt{mjiang@cc.usu.edu}.
Supported in part by NSF grant DBI-0743670.
}
\and Csaba D. T\'oth\footnote{Department of Mathematics,
University of Calgary, AB, Canada T2N~1N4. 
E-mail: \texttt{cdtoth@ucalgary.ca}.
Supported in part by NSERC grant RGPIN 35586.}
}

\begin{document}

\maketitle

\begin{abstract}
The Fermat-Weber center of a planar body $Q$ is a point in the plane
from which the average distance to the points in $Q$ is minimal.
We first show that for any convex body $Q$ in the plane, the average
distance from the  Fermat-Weber center of $Q$ to the points of $Q$ is
larger than $\frac{1}{6} \cdot \Delta(Q)$, where $\Delta(Q)$ is the diameter of $Q$.
This proves a conjecture of Carmi, Har-Peled and Katz.
From the other direction, we prove that the same average distance
is at most $\frac{2(4-\sqrt3)}{13} \cdot \Delta(Q) < 0.3490 \cdot \Delta(Q)$.
The new bound substantially improves the previous bound of
$\frac{2}{3 \sqrt3} \cdot \Delta(Q) \approx 0.3849 \cdot \Delta(Q)$ due to
Abu-Affash and Katz, and brings us closer to the conjectured value of
$\frac{1}{3} \cdot \Delta(Q)$.
We also confirm the upper bound conjecture for centrally symmetric
planar convex bodies.
\end{abstract}

\section{Introduction} \label{sec:intro}

The Fermat-Weber center of a measurable planar set $Q$ with positive area
is a point in the plane that minimizes the average distance to the points in $Q$.
Such a point is the ideal location for a base station (e.g., fire station
or a supply station) serving the region $Q$, assuming the region has
uniform density. Given a measurable set $Q$ with positive area and a
point $p$ in the plane, let $\mu_Q(p)$ be the average distance between
$p$ and the points in $Q$, namely,
$$ \mu_Q(p)= \frac{\int_{q \in Q} {\rm dist}(p,q) \intd q}{{\rm area}(Q)}, $$
where ${\rm dist}(p,q) = |pq| $ is the Euclidean distance between $p$ and
$q$.   
Let $FW_Q$ be the Fermat-Weber center of $Q$, and write
$\mu^*_Q = \min \{\mu_Q(p): p\in \RR^2\} =\mu_Q(FW_Q)$.

Carmi, Har-Peled and Katz~\cite{CHK05} showed that there exists a constant
$c>0$ such that $\mu^*_Q \geq c \cdot \Delta(Q) $  holds for any convex body $Q$,
where $\Delta(Q)$ denotes the diameter of $Q$. The convexity is necessary,
since it is easy to construct nonconvex regions where the average
distance from the Fermat-Weber center is arbitrarily small compared to
the diameter. Of course the opposite inequality  $\mu^*_Q \leq c'
\cdot \Delta(Q) $ holds for any body $Q$ (convexity is not required),
since we can trivially take $c'=1$.

Let $c_1$ denote the infimum, and $c_2$ denote the supremum of
$\mu^*_Q/\Delta(Q)$ over all convex bodies $Q$ in the plane.
Carmi, Har-Peled and Katz~\cite{CHK05} conjectured that
$c_1 =\frac{1}{6}$ and $c_2=\frac{1}{3}$. Moreover, they conjectured
that the supremum $c_2$ is attained for a circular disk $D$,
where $\mu^*_D=\frac{1}{3}\cdot \Delta(D)$. They also proved that
$\frac{1}{7} \leq c_1 \leq \frac{1}{6}$.
The inequality $c_1 \leq \frac{1}{6}$ is given by an infinite
sequence of rhombi, $P_\eps$, where one diagonal has some
fixed length, say $2$, and the other diagonal tends to zero; see Fig.~\ref{f1}.
By symmetry, the Fermat-Weber center of a rhombus is its center of
symmetry, and one can verify that $\mu^*_{P_\eps}/\Delta(P_\eps)$
tends to $\frac{1}{6}$. The lower bound for $c_1$ has been recently
further improved by Abu-Affash and Katz from $\frac{1}{7}$ to
$\frac{4}{25}$~\cite{AK08}. Here we establish that $c_1 =\frac{1}{6}$
and thereby confirm the first of the two conjectures of Carmi, Har-Peled and Katz.

\begin{figure} [htb]
\centerline{\epsfxsize=.7\textwidth \epsffile{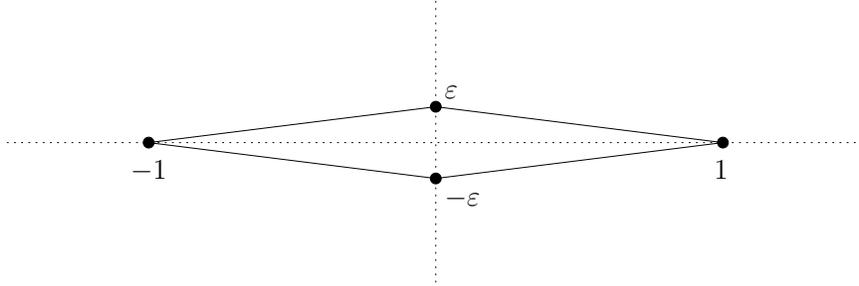}}
\caption{\small A flat rhombus $P_\eps$, with
$ \lim_{\eps \rightarrow 0} \mu^*_{P_\eps}/\Delta(P_\eps) = \frac{1}{6}$.}
\label{f1}
\end{figure}

Regarding the second conjecture, recently Abu-Affash and Katz
proved that $c_2 \leq \frac{2}{3 \sqrt3} = 0.3849\ldots$. Here we further
improve this bound and bring it closer to the conjectured value of $\frac{1}{3}$.
Finally, we also confirm the upper bound conjecture for centrally
symmetric convex bodies $Q$.

Our main results are summarized in the following two theorems:

\begin{theorem}\label{T1}
For any convex body $Q$ in the plane, we have $\mu^*_Q > \frac{1}{6} \cdot \Delta(Q)$.
\end{theorem}

\begin{theorem}\label{T2}
For any convex body $Q$ in the plane, we have
$$ \mu^*_Q \leq \frac{2(4-\sqrt3)}{13} \cdot \Delta(Q) < 0.3490 \cdot \Delta(Q). $$
Moreover, if $Q$ is  centrally symmetric, then $\mu^*_Q \leq
\frac{1}{3} \cdot \Delta(Q)$.
\end{theorem}

\paragraph{Remarks.}

\noindent {\bf 1.}
The average distance from a point $p$ in the plane can be defined analogously
for finite point sets and for rectifiable curves. Observe that for a line segment
$I$ (a one-dimensional convex set), we would have $\mu^*_I/\Delta(I) = \frac{1}{4}$.
It might be interesting to note that while the thin rhombi mentioned
above tend in the limit to a line segment, the value of the
limit $\mu^*_{P_\eps}/\Delta(P_\eps)$ equals $\frac{1}{6}$, not $\frac{1}{4}$.

\medskip
\noindent {\bf 2.}
In some applications, the cost of serving a location $q$
from a facility at point $p$ is ${\rm dist}^\kappa(p,q)$ for some
exponent $\kappa\geq 1$, rather than ${\rm dist}(p,q)$. We can define
$\mu_Q^\kappa(p)=\left(\int_{q\in Q} {\rm dist}^\kappa(p,q)\intd
q\right)/{\rm area}(Q)$
and  $\mu_Q^{\kappa *}=\inf \{\mu_Q^\kappa(p): p \in \RR^2\}$, which is
invariant under congruence. The ratio $\mu_Q^{\kappa *} /
\Delta^\kappa(Q)$ is also invariant under similarity. The proof of
Theorem~\ref{T1} carries over for this variant and shows that
$\mu_Q^{\kappa*}/\Delta^\kappa(Q)>\frac{1}{(\kappa+2)2^\kappa}$ for any
convex body $Q$, and $\lim_{\eps\rightarrow 0} \mu_{P_\eps}^{\kappa
*}/2^\kappa
=\frac{1}{(\kappa+2)2^\kappa}$.
For the upper bound,  the picture is not so clear: $\mu_Q^*/\Delta(Q)$ is
conjectured to be maximal for the circular disk, however,
there is a $\kappa\geq 1$ such that $\mu_Q^{\kappa *}/\Delta^\kappa(Q)$
cannot be maximal for the disk. In particular, if $D$ is a disk of diameter
2 and $R$ is a convex body of diameter $2$ whose smallest enclosing disk
has diameter more than $2$ (e.g., a regular or a Reuleaux triangle of
diameter $2$), then $\mu_D^{\kappa *} <\mu_{R}^{\kappa *}$,
for a sufficiently large $\kappa>1$.
Let $o$ be an arbitrary point in the plane, and let
$D$ be centered at $o$.
Then $\int_{q\in D}{\rm dist}^\kappa(o,q)\intd q =
\int_0^{2\pi} \int_0^1 r^\kappa \cdot r \intd r \intd \theta =
\frac{2\pi}{\kappa+2}$,
and so $\lim_{\kappa\rightarrow \infty}\mu_D^{\kappa *} \leq
\lim_{\kappa\rightarrow \infty} \frac{2}{\kappa+2}= 0$.
On the other hand, for any region $R'$ lying outside of $D$ and
for any $\kappa\geq 1$, we have $\int_{q\in R'} {\rm dist}^\kappa(o,q)
\intd q\geq {\rm area}(R')>0$.  If $R'=R\setminus D$ is the part of $R$ lying
outside $D$, then
$\lim_{\kappa\rightarrow \infty}\mu_{R}^{\kappa *} \geq {\rm
area}(R')/\pi>0$.

\paragraph{Related work.}
Fekete, Mitchell, and Weinbrecht~\cite{FMW00} studied a continuous version
of the problem for polygons with holes, where the distance between two
points is measured by the $L_1$ geodesic distance.
A related question on Fermat-Weber centers in a discrete setting
deals with stars and Steiner stars~\cite{DTX09,FM00}.
\later{
A Steiner star for a set $S$ of $n$ points in $\RR^2$ connects an
arbitrary  center point to all points of $S$, while a
star connects a point $p\in S$ to the remaining $n-1$ points of
$S$. All connections are realized by straight line segments.
The maximum ratio between the lengths of the minimum star and the
minimum Steiner star, over all finite point configurations in $\RR^2$,
is called the star Steiner ratio in $\RR^2$, denoted by $\rho$.
Fekete and Meijer~\cite{FM00}, who were the first to study the star Steiner
ratio, proved that $\rho \leq \sqrt{2}$. The current best upper bound
is (about) $1.3631$, see~\cite{DTX09}.
} 
The reader can find more information on other variants of the
Fermat-Weber problem in~\cite{DKSW02,W93}.

\section{Lower bound: proof of Theorem~\ref{T1}} \label{sec:T1}

In a nutshell the proof goes as follows.
Given a convex body $Q$, we take its Steiner symmetrization with
respect to a supporting line of a diameter segment $cd$, followed by another
Steiner symmetrization with respect to the perpendicular bisector of
$cd$. The two Steiner symmetrizations preserve the area and the
diameter, and do not increase the average distance from the
corresponding Fermat-Weber centers.
In the final step, we prove that the inequality holds for
a convex body with two orthogonal symmetry axes.

\paragraph{Steiner symmetrization with respect to an axis.}
Steiner symmetrization of a convex figure $Q$ with respect to an axis
(line) $\ell$ consists in replacing $Q$ by a new figure $S(Q,\ell)$ with
symmetry axis $\ell$ by means of the following construction:
Each chord of $Q$ orthogonal to $\ell$ is displaced along its line to
a new position where it is symmetric with respect to $\ell$,
see~\cite[pp.~64]{YB61}. The resulting figure $S(Q,\ell)$ is also
convex, and obviously has the same area as $Q$.

A body $Q$ is $x$-monotone if the intersection of $Q$ with every
vertical line is either empty or is connected (that is, a point
or a line segment).
Every $x$-monotone body $Q$ is bounded by the graphs of some functions
$f: [a,b]\rightarrow \RR$ and $g: [a,b]\rightarrow \RR$ such that
$g(x)\leq f(x)$ for all $x\in [a,b]$. The Steiner symmetrization with
respect to the $x$-axis $\ell_x$ transforms $Q$ into an $x$-monotone body
$S(Q,\ell_x)$ bounded by the functions $\frac{1}{2}(f(x)-g(x))$ and
$\frac{1}{2}(g(x)-f(x))$ for $x\in [a,b]$.
As noted earlier, ${\rm area}(S(Q,\ell_x))={\rm area}(Q)$.
The next two lemmas do not require the convexity of $Q$.

\begin{lemma} \label{L1}
Let $Q$ be an $x$-monotone body in the plane with a diameter parallel or
orthogonal to the $x$-axis, then $\Delta(Q)=\Delta(S(Q,\ell_x))$.
\end{lemma}
\begin{proof}
Let $Q'=S(Q,\ell_x)$. If $Q$ has a diameter parallel to the $x$-axis,
then the diameter is $[(a,c), (b,c)]$, with a value $c\in \RR$,
$g(a)= c= f(a)$ and $g(b)= c= f(b)$. That is, $\Delta(Q)=b-a$.
In this case, the diameter of $Q'$ is at least $b-a$, since both points
$(a,0)$ and $(b,0)$ are in $Q'$. If $Q$ has a diameter orthogonal to
the $x$-axis, then the diameter is
$[(x_0,f(x_0)), (x_0,g(x_0))]$ for some $x_0\in [a,b]$, and
$\Delta(Q)=f(x_0)-g(x_0)$. In this case, the diameter of $Q'$ is at least
$f(x_0)-g(x_0)$, since both points $(x_0,\frac{1}{2}(f(x_0)-g(x_0)))$ and
$(x_0,\frac{1}{2}(g(x_0)-f(x_0)))$ are in $Q'$. Therefore, we have
$\Delta(Q')\geq \Delta(Q)$.

Let $A_1$ and $A_2$ be two points on the boundary of $Q'$ such that
$\Delta(Q')= {\rm dist}(A_1,A_2)$. 
Since $Q'$ is symmetric to the $x$-axis, points $A_1$ and $A_2$ cannot
both be on the upper (resp., lower)  
boundary of $Q'$. Assume w.l.o.g.\ that $A_1=(x_1,\frac{1}{2}(f(x_1)-g(x_1)))$ and
$A_2=(x_2,\frac{1}{2}(g(x_2)-f(x_2)))$ for some $a\leq x_1,x_2\leq b$. 
$$\Delta(Q')= {\rm dist}(A_1,A_2)= \sqrt{(x_2-x_1)^2 +
\left(\frac{f(x_1)+f(x_2)-g(x_1)-g(x_2)}{2}\right)^2}.$$
Now consider the following two point pairs in $Q$. The distance between
$B_1=(x_1,f(x_1))$ and $B_2=(x_2,g(x_2))$ is ${\rm dist}(B_1,B_2)=
\sqrt{(x_2-x_1)^2 + (f(x_1)-g(x_2))^2}$.
Similarly, the distance between $C_1=(x_1,g(x_1))$ and $C_2=(x_2,f(x_2))$ is ${\rm
dist}(C_1,C_2) = \sqrt{(x_2-x_1)^2 + (g(x_1)-f(x_2))^2}$.
Using the inequality between the arithmetic and quadratic means, we have
$$ \left(\frac{f(x_1)+f(x_2)-g(x_1)-g(x_2)}{2}\right)^2 \leq
\frac{(f(x_1)-g(x_2))^2 + (g(x_1)-f(x_2))^2}{2}.$$
This implies that ${\rm dist}(A_1,A_2) \leq \max  ( {\rm dist}(B_1,B_2),
{\rm dist}(C_1,C_2))$, and so $\Delta(Q')\leq \Delta(Q)$.
We conclude that $\Delta(Q)=\Delta(S(Q,\ell_x))$.
\end{proof}

\begin{lemma} \label{L2}
If $Q$ is an $x$-monotone body in the plane, then $\mu_Q^*\geq
\mu_{S(Q,\ell_x)}^*$.
\end{lemma}
\begin{proof}
If $(x_0,y_0)$ is the Fermat-Weber center of $Q$, then
$$\mu_Q^*= \frac{\int_a^b \int_{g(x)}^{f(x)} \sqrt{(x-x_0)^2+(y-y_0)^2} \intd y
\intd x}{{\rm area}(Q)}.$$
Observe that $\int_{g(x)}^{f(x)} \sqrt{(x-x_0)^2+(y-y_0)^2} \intd y$ is the
integral of the distances of the points in a line segment of length
$f(x)-g(x)$ from a point at distance $|x-x_0|$ from the supporting
line of the segment. This integral is minimal if the point is on the
orthogonal bisector of the segment. That is, we have
\begin{eqnarray}
\int_{g(x)}^{f(x)} \sqrt{(x-x_0)^2+(y-y_0)^2} \intd y
&\geq& \int_{g(x)}^{f(x)}
\sqrt{(x-x_0)^2+\left(y-\frac{f(x)-g(x)}{2}\right)^2} \ dy \nonumber\\
&=&  \int_{\frac{1}{2}(g(x)-f(x))}^{\frac{1}{2}(f(x)-g(x))}
\sqrt{(x-x_0)^2+y^2} \intd y.\nonumber
\end{eqnarray}
Therefore, we conclude that
\begin{eqnarray}
\mu_Q^*
& = & \frac{\int_a^b \int_{g(x)}^{f(x)} \sqrt{(x-x_0)^2+(y-y_0)^2} \intd y
\intd x}{{\rm area}(Q)}\nonumber\\
&\geq& \frac{\int_a^b \int_{\frac{1}{2}(g(x)-f(x))}^{\frac{1}{2}(f(x)-g(x))}
\sqrt{(x-x_0)^2+y^2} \intd y \intd x}{{\rm area}(S(Q,x))} =
\mu_{S(Q,\ell_x)}((x_0,0))\geq \mu_{S(Q,\ell_x)}^*.\nonumber 
\end{eqnarray}
\end{proof}

\paragraph{Triangles.} We next consider right triangles of a special
kind, lying in the first quadrant, and show that the average distance from
the origin to their points is larger than $\frac{1}{3}$.

\begin{lemma} \label{L3}
Let $T$ a right triangle in the first quadrant based on the $x$-axis,
with vertices $(a,0)$, $(a,b)$, and $(1,0)$, where $0 \leq a < 1$, and $b>0$.
Then $\mu_T(o) > \frac{1}{3}$.
\end{lemma}
\begin{proof}
We use the simple fact that the $x$-coordinate of a point is a lower
bound to the distance from the origin.
\begin{eqnarray*}
\mu_T(o) &=&
\frac{\int_a^1 (\int_0^{b(1-x)/(1-a)} \sqrt{x^2 + y^2} \intd y) \intd x}
{b(1-a)/2} >
\frac{\int_a^1 (\int_0^{b(1-x)/(1-a)} x \intd y) \intd x} {b(1-a)/2} \\
&=&
\frac{\frac{b}{1-a} \int_a^1 x(1-x) \intd x} {b(1-a)/2} =
\frac{2}{(1-a)^2}\left(\frac{x^2}{2} - \frac{x^3}{3}\right) \Big{|}_a^1\\
&=& \frac{2}{(1-a)^2} \cdot \frac{(2a^3-3a^2+1)}{6} =
\frac{2}{(1-a)^2} \cdot \frac{(1-a)(1+a-2a^2)}{6} \\
&=& \frac{1}{(1-a)} \cdot \frac{(1+a-2a^2)}{3} \geq \frac{1}{3}.
\end{eqnarray*}
The last inequality in the chain follows from $0 \leq a<1$.
The inequality in the lemma is strict, since $\sqrt{x^2 + y^2} > x$ for all points
above the $x$-axis.
\end{proof}

\begin{corollary} \label{C1}
Let $P$ be any rhombus. Then $\mu^*_P > \frac{1}{6} \cdot \Delta(P)$.
\end{corollary}
\begin{proof}
Without loss of generality, we may assume that $P$ is symmetric with
respect to both the $x$-axis and the $y$-axis. Let us denote the vertices of $P$ by
$(-1,0)$, $(1,0)$, $(0,-b)$, and $(0,b)$, where $b \leq 1$. We have $\Delta(P)=2$.
By symmetry, $\mu^*_P$ equals the average distance between the origin $(0,0)$ and
the points in one of the four congruent right triangles forming $P$.
Consider the triangle $T$ in the first quadrant.
By Lemma~\ref{L3} (with $a=0$), we have $\mu^*_P = \mu_T(o) > \frac{1}{3}$.
Since $\Delta(P)=2$, we have $\mu^*_P > \frac{1}{6} \cdot \Delta(P)$, as desired.
\end{proof}

\begin{lemma} \label{L4}
Let $T$ be a triangle in the first quadrant with a vertical side
on the line $x=a$, where $0 \leq a <1$, and a third vertex at
$(1,0)$. Then $\mu_T(o) > \frac{1}{3}$.
\end{lemma}
\begin{proof}
Refer to Fig.~\ref{f2}(ii). Let $U$ be a right triangle obtained from
$T$ by translating each vertical chord of $T$ down until its lower
endpoint is on the $x$-axis. Note that ${\rm area}(T)={\rm area}(U)$.
Observe also that the average distance from the origin decreases in this
transformation, namely $ \mu_T(o) \geq \mu_U(o)$.
By Lemma~\ref{L3}, we have $\mu_U(o) > \frac{1}{3}$, and so
$\mu_T(o) > \frac{1}{3}$, as desired.
\end{proof}

\smallskip
We now have all necessary ingredients to prove Theorem~\ref{T1}.

\paragraph {Proof of Theorem~\ref{T1}.}
Refer to Fig.~\ref{f2}.
Let $Q$ be a convex body in the plane, and let $c,d\in Q$ be two
points at $\Delta(Q)$ distance apart. We may assume that $c=(-1,0)$
and $d=(1,0)$, by a similarity transformation if necessary, so that
$\Delta(Q)=2$ (the ratio $\mu^*_Q/\Delta(Q)$ is invariant under similarities).
Apply a Steiner symmetrization with respect to the $x$-axis, and then a second
Steiner symmetrization with respect to the $y$-axis.
The resulting body $Q'=S(S(Q,\ell_x), \ell_y)$ is convex, and it is
symmetric with respect to both coordinate axes.
We have $\Delta(Q')=\Delta(Q)=2$ by Lemma~\ref{L1}, and in fact $c,d\in Q'$.
We also have $\mu^*_{Q'}\leq \mu^*_Q$ by Lemma~\ref{L2}.
\begin{figure} [b]
\centerline{\epsfxsize=.97\textwidth \epsffile{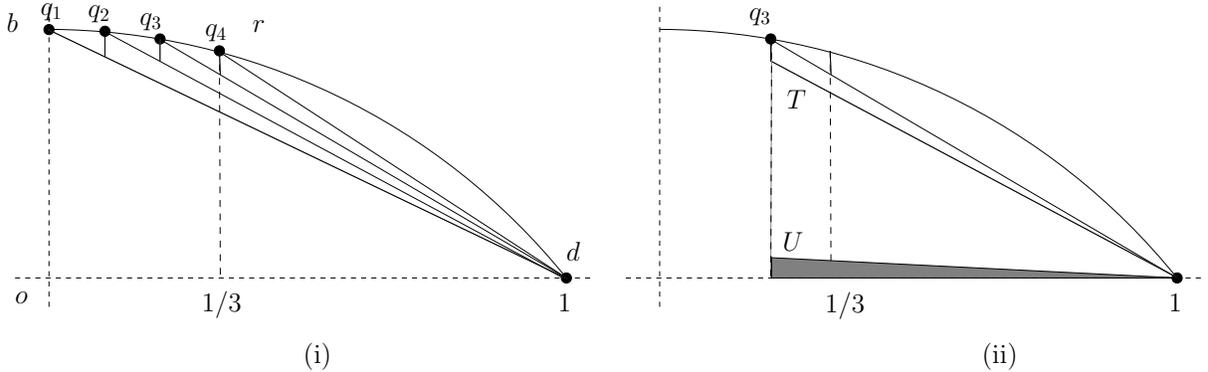}}
\caption{\small (i) The subdivision of $Q_1$ for $n=3$. Here
$o=(0,0)$, $q_1=b=(0,h)$, $q_4=r$, $d=(1,0)$.
(ii) Transformation in the proof of Lemma~\ref{L4}.}
\label{f2}
\end{figure}

Let $Q_1$ be the part of $Q'$ lying in the first quadrant:
$Q_1 =\{(x,y) \in Q': x,y \geq 0\}$.
By symmetry, $FW_{Q'}=o$ and we have $\mu^*_{Q'}=\mu_{Q'}(o)=\mu_{Q_1}(o)$.
Let $\gamma$ be the portion of the boundary of $Q'$ lying in the first quadrant,
between points $b=(0,h)$, with $0<h\leq 1$, and $d=(1,0)$. For any two
points $p,q\in \gamma$ along $\gamma$, denote by $\gamma(p,q)$ the
portion of $\gamma$ between $p$ and $q$.
Let $r$ be the intersection point of $\gamma$ and the vertical line $x=\frac{1}{3}$.

For a positive integer $n$, subdivide $Q_1$ into at most $2n+2$ pieces as follows.
Choose $n+1$ points $b=q_1,q_2\ldots , q_{n+1}=r$ along $\gamma(b,r)$
such that $q_i$ is the intersection of $\gamma$ and the vertical line $x=(i-1)/3n$.
Connect each of the $n+1$ points to $d$ by a straight line segment. These segments
subdivide $Q_1$ into $n+2$ pieces: the right triangle $T_0=\Delta bod$; a
convex body $Q_0$ bounded by $rd$ and $\gamma(r,d)$; and $n$ curvilinear triangles
$\Delta q_idq_{i+1}$ for $i=1,2,\ldots , n$. For simplicity, we assume
that neither $Q_0$, nor any of the curvilinear triangles are degenerate;
otherwise they can be safely ignored (they do not contribute
to the value of $\mu^*_{Q'}$). Subdivide each curvilinear triangle
$\Delta q_idq_{i+1}$ along the vertical line through $q_{i+1}$ into
a small curvilinear triangle $S_i$ on the left and a triangle $T_i$ incident
to point $d$ on the right. The resulting subdivision has $2n+2$ pieces,
under the nondegeneracy assumption.

By Lemma~\ref{L3}, we have $\mu_{T_0}(o)>\frac{1}{3}$. Observe that the difference
$\mu_{T_0}(o)-\frac{1}{3}$ does not depend on $n$, and let $\delta
=\mu_{T_0}(o)-\frac{1}{3}$.
By Lemma~\ref{L4}, we also have  $\mu_{T_i}(o)>\frac{1}{3}$, for
each $i=1,2,\ldots ,n$.
Since every point in $Q_0$ is at distance at least $\frac{1}{3}$ from the origin,
we also have $\mu_{Q_0}(o)\geq \frac{1}{3}$.

For the $n$ curvilinear triangles $S_i$, $i=1,2,\ldots ,n$, we use the
trivial lower bound $\mu_{S_i}(o)\geq 0$. We now show that their total area
$s_n=\sum_{i=1}^n {\rm area}(S_i)$ tends to 0 if $n$ goes to infinity.
Recall that the $y$-coordinates of the points $q_i$ are at most 1, and
their $x$-coordinates are at most $\frac{1}{3}$. This implies that the slope
of every line $q_id$, $i=1,2,\ldots , n+1$,
is in the interval $[-3/2, 0]$. Therefore, $S_i$ is contained in a right triangle
bounded by a horizontal line through $q_i$, a vertical line through $q_{i+1}$, and
the line $q_id$. The area of this triangle is at most $\frac{1}{2}(\frac{1}{3n}\cdot
(\frac{3}{2}\cdot \frac{1}{3n})) = 1/(12n^2)$.
That is, $s_n =\sum_{i=1}^n {\rm area}(S_i) \leq 1/(12n)$.
In particular, $s_n \leq \delta \cdot {\rm area}(T_0)$ for a sufficiently large $n$.
Then we can write
\begin{eqnarray*}
\mu_{Q_1}(o) &=& \frac{\int_{p\in Q_1}{\rm dist}(o,p) \intd p}{{\rm area}(Q_1)}
\geq \frac{ \mu_{Q_0}(o) \cdot {\rm area}(Q_0)+
\sum_{i=0}^n \mu_{T_i}(o) \cdot {\rm area}(T_i)}{{\rm area}(Q_1)} \\
&\geq& \frac{\frac{1}{3}({\rm area}(Q_1)-s_n)+\delta \cdot {\rm area}(T_0)}
{{\rm area}(Q_1)} 
\geq \frac{1}{3}+ \frac{2\delta \cdot {\rm area}(T_0)} {3 \cdot {\rm area}(Q_1)}
>\frac{1}{3}.
\end{eqnarray*}
This concludes the proof of Theorem~\ref{T1}.
\endproof

\medskip
\noindent {\bf Remark.} A finite triangulation, followed by
taking the limit suffices to prove the slightly weaker, non-strict
inequality: $\mu^*_Q \geq \frac{1}{6} \cdot \Delta(Q)$.

\section{Upper bounds: proof of Theorem~\ref{T2}} \label{sec:T2}

Let $Q$ be a planar convex body and let $D=\Delta(Q)$.
Let $\partial Q$ denote the boundary
of $Q$, and let ${\rm int}(Q)$ denote the interior of $Q$.
Let $\Omega$ be the smallest disk enclosing $Q$, and let $o$ and 
 $R$ be the center and respectively the radius of $\Omega$. 
Write $a= \frac{2(4-\sqrt3)}{13}$.
By the convexity of $Q$, $o \in Q$, as observed in~\cite{AK08}.
Moreover, Abu-Affash and Katz~\cite{AK08} have shown that the average
distance from $o$ to the points in $Q$ satisfies
$$ \mu_Q(o) \leq \frac{2}{3 \sqrt3} \cdot \Delta(Q) < 0.3850 \cdot \Delta(Q). $$

Here we further refine their analysis and derive a better upper bound
on the average distance from $o$ to the points in $Q$:
$$ \mu_Q(o) \leq \frac{2(4-\sqrt3)}{13} \cdot \Delta(Q) < 0.3490 \cdot \Delta(Q). $$
Since the average distance from the Fermat-Weber center of $Q$ is not
larger than that from $o$, we immediately get the same upper bound on $c_2$.
We need the next simple lemma established in~\cite{AK08}. Its
proof follows from the definition of average distance.

\begin{lemma} {\rm~\cite{AK08}.} \label{L-AK}
Let $Q_1$, $Q_2$ be two (not necessarily convex) disjoint bodies in
the plane, and $p$ be a point in the plane. Then
$\mu_{(Q_1 \cup Q_2)}(p) \leq \max (\mu_{Q_1}(p), \mu_{Q_2}(p))$.
\end{lemma}

By induction, Lemma~\ref{L-AK} yields:

\begin{lemma} \label{L5}
Let $Q_1, Q_2, \ldots, Q_n$ be $n$ (not necessarily convex) pairwise
disjoint bodies in the plane, and $p$ be a point in the plane. Then
$$ \mu_{(Q_1 \cup \ldots \cup Q_n)}(p) \leq
\max (\mu_{Q_1}(p), \ldots \mu_{Q_n}(p)). $$
\end{lemma}

We also need the following classical result of Jung~\cite{J10};
see also~\cite{HD64}.

\begin{theorem}  {\rm (Jung~\cite{J10}).} \label{T-J}
Let $S$ be a set of diameter $\Delta(S)$ in the plane. Then
$S$ is contained in a circle of radius $\frac{1}{\sqrt3} \cdot \Delta(S)$.
\end{theorem}

By Theorem~\ref{T-J} we have 

\begin{equation}\label{eq:RD}
\frac12 D \le R \le \frac1{\sqrt3} D.
\end{equation}

Observe that the average distance from the center
of a circular sector of radius $r$ and center angle $\alpha$ to the
points in the sector is
\begin{equation} \label{E1}
\frac{\int_0^r \alpha x^2 \intd x}{\int_0^r \alpha x \intd x}=
\frac{\alpha r^3/3}{\alpha r^2/2} = \frac{2r}{3}.
\end{equation}

\paragraph {Proof of Theorem~\ref{T2}.}
If $o \in \partial Q$ then $Q$ is contained in a
halfdisk $\Theta$ of $\Omega$, of the same diameter $D$, with $o$ as
the midpoint of this diameter. Then by \eqref{E1}, it follows that 
$ \mu_Q(o) \leq \frac{1}{3} \cdot D$, as required. 

We can therefore assume that $o \in {\rm int}(Q)$. 
Let $\eps>0$ be sufficiently small.
For a large positive integer $n$, subdivide $\Omega$
into $n$ congruent circular double sectors (wedges) $W_1,\ldots,W_n$,
symmetric about $o$ (the center of $\Omega$), where each sector
subtends an angle $\alpha=\pi/n$. Consider a double sector $W_i =U_i \cup V_i$,
where $U_i$ and $V_i$ are circular sectors of $\Omega$.
Let $X_i \subseteq U_i$, and $Y_i \subseteq V_i$
be two minimal circular sectors centered at $o$ and
containing $U_i \cap Q$, and $V_i \cap Q$, respectively:
$ U_i \cap Q \subseteq X_i$, and $ V_i \cap Q \subseteq Y_i$.
Let $x_i$ and $y_i$ be the radii of $X_i$ and $Y_i$, respectively.
Let $X'_i \subseteq X_i$, and $Y'_i \subseteq Y_i$ be two circular
subsectors of radii $(1-\eps)x_i$ and $(1-\eps)y_i$, respectively.
Since $o \in {\rm int}(Q)$, we can select $n=n(Q,\eps)$ large enough, so
that for each $1 \leq i \leq n$, the subsectors $X'_i$ and $Y'_i$
are nonempty and entirely contained in $Q$. That is, for every $i$, we have
\begin{equation} \label{E2}
X'_i \cup Y'_i \subseteq W_i \cap Q \subseteq X_i \cup Y_i.
\end{equation}

It is enough to show that for any double sector $W=W_i$, we have
$$ \lim_{\eps \to 0} \mu_{(W \cap Q)}(o) \leq a D, $$
since then, Lemma~\ref{L5} (with $W_i$ being the $n$ pairwise disjoint
regions) will imply that $ \mu_Q(o) \leq a D$, concluding the proof of
Theorem~\ref{T2}. For simplicity, write $x=x_i$, and $y=y_i$.
Obviously the diameter of $W \cap Q$ is at most $D$, hence  $x+y \leq D$.
We can assume w.l.o.g.\ that $y \leq x$, so by Theorem~\ref{T-J}
we also have $x \leq \frac{1}{\sqrt3} \cdot D $. Hence so far, our 
constraints are:

\begin{equation} \label{E3}
0< y \leq x \leq \frac{1}{\sqrt3} \cdot D
\hspace{1cm}\mbox{\rm and}\hspace{1cm} x+y \leq D.
\end{equation}

By the minimality of the disk $\Omega$, the convex body $Q$
either contains three points $q_1,q_2,q_3$ on the boundary of $\Omega$
such that the triangle $q_1q_2q_3$ contains the disk center $o$ in the interior,
or contains two points $q_1,q_2$ on the boundary of $\Omega$
such that the segment $q_1q_2$ goes through the disk center $o$.
In the latter case, the segment $q_1q_2$ can be viewed as a degenerate triangle
$q_1q_2q_3$ with two coinciding vertices $q_2$ and $q_3$.

Let $r$ be the radius of the largest disk centered at $o$
that is contained in the convex body $Q$.
Then $r$ is at least the distance from $o$ to the longest side of the triangle
$q_1q_2q_3$, say $q_1q_2$.
Since $|q_1q_2| \le D$, $|o q_1| = |o q_2| = R$,
we have
$$
r \ge \sqrt{R^2 - D^2/4}.
$$
Then the constraints in~\eqref{E3} can be expanded to the following:
\begin{equation}\label{eq:constraints-new}
\sqrt{R^2 - D^2/4} \le y \le x \le R \le D/\sqrt3
\quad\textup{and}\quad
x + y \le D.
\end{equation}

By the definition of average distance, we can write
\begin{eqnarray} \label{E4}
\mu_{(W \cap Q)}(o) &=&
\frac{\int_{p\in (W \cap Q)}{\rm dist}(o,p) \intd p}{{\rm area}(W \cap Q)}\nonumber\\
&\leq& \dfrac{\alpha \cdot \frac{x^2}{2} \cdot \frac{2x}{3} +
\alpha \cdot \frac{y^2}{2} \cdot \frac{2y}{3} }
{\alpha (1-\eps)^2 \cdot \left(\frac{x^2}{2} + \frac{y^2}{2}\right)} 
= \frac{2}{3} \cdot \frac{x^3+y^3}{(1-\eps)^2 \cdot (x^2+y^2)}.
\end{eqnarray}
Let
\begin{equation}\label{eq:f}
f(x,y)= \frac{2}{3} \cdot \frac{x^3+y^3}{x^2+y^2},
\textrm{ \ and \ }
f_1(x,y,\eps)= \frac{2}{3} \cdot \frac{x^3+y^3}{(1-\eps)^2 \cdot (x^2+y^2)}.
\end{equation}
Clearly for any feasible pair $(x,y)$, we have
$$
\lim_{\eps \to 0} f_1(x,y,\eps) = f(x,y).
$$

It remains to maximize $f(x,y)$ subject to the constraints in
\eqref{eq:constraints-new}. We will show that under these constraints,
\begin{equation}\label{E10}
f(x,y) \leq \frac{2(4-\sqrt3)}{13} \cdot D.
\end{equation}
Then
$$
\lim_{\eps \to 0} \mu_{(W \cap Q)}(o) \leq 
\lim_{\eps \to 0} f_1(x,y,\eps) = f(x,y)
\leq \frac{2(4-\sqrt3)}{13} \cdot D,
$$
as required.

We next verify the upper bound in~\eqref{E10}.
Throughout our analysis,
we may assume that $D$ is a fixed constant and 
$x$, $y$, and $R$ are variable parameters.
Substituting $z = y/x$ in~\eqref{eq:f}, we have
$$
f(x, y) = g(x, z) = \frac{2x}3 \cdot \frac{1 + z^3}{1 + z^2}.
$$
Then, taking the partial derivative of $g(x,z)$ with respect to $z$, we have

\begin{align*}
\frac{\partial}{\partial z} g(x, z) &= \frac{2x}3 \cdot
        \left( \frac{3 z^2}{1 + z^2} - \frac{1 + z^3}{(1 + z^2)^2}
2z\right) \\
&= \frac{2x}3 \cdot \frac{3 z^2 (1 + z^2) - (1 + z^3) 2z}{(1 + z^2)^2}
= \frac{2x}3 \cdot \frac{z (z^3 + 3 z - 2)}{(1 + z^2)^2}.
\end{align*}

The cubic equation $z^3 + 3 z - 2 = 0$ has exactly one real root
$z_0 = (\sqrt2 + 1)^{1/3} - (\sqrt2 - 1)^{1/3} = 0.596\ldots.$
Thus for a fixed $x$, the function $g(x, z)$ is strictly decreasing for
$0 \le z \le z_0$ and is strictly increasing for $z_0 \le z \le 1$.
Therefore,
by the upper bound that
$x + y \le D$
and the lower bound that
$\sqrt{R^2 - D^2/4} \le r \le y$
in~\eqref{eq:constraints-new},
the function $f(x, y)$ is maximized when $y$ takes one of the following
two extreme values:
$$
y_1 = \sqrt{R^2 - D^2/4}
\quad\textup{and}\quad
y_2 = D - x.
$$
By the inequality that $x \le R \le D/\sqrt3$ in~\eqref{eq:constraints-new},
it follows that
$x + y_1 \le R + \sqrt{R^2 - D^2/4} \le D/\sqrt3 + D/\sqrt{12} < D$.
Since
$x + y_2 = D$,
we have $y_1 < y_2$.

\paragraph{Case 1.}

We first consider the easy case that $y = y_2$. Then $x + y = D$,
and we have
$$
f(x, y) = \frac23 \cdot \frac{x^3 + y^3}{x^2 + y^2}
= \frac23 \cdot \frac{(x+y)^3 - 3(x+y)xy}{(x+y)^2 - 2xy}
= \frac23 \cdot \frac{D^3 - 3Dxy}{D^2 - 2xy}.
$$
Substituting $w = xy$,
we tranform the function $f(x,y)$ to a function $h_1(w)$:
$$
f(x, y) = h_1(w) = \frac23 \cdot \frac{3Dw - D^3}{2w - D^2}.
$$

The function $h_1(w)$ is decreasing in $w$ because
\begin{align*}
\frac{\dd}{\dd w} h_1(w) &= \frac23 \cdot
        \left( \frac{3D}{2w - D^2} - \frac{2(3Dw - D^3)}{(2w - D^2)^2}
\right) \\
&= \frac23 \cdot \frac{3D(2w - D^2) - 2(3Dw - D^3)}{(2w - D^2)^2}
= \frac23 \cdot \frac{-D^3}{(2w - D^2)^2} \le 0.
\end{align*}

Thus $f(x,y)$ is maximized when $xy$ is minimized.
With the sum $x + y$ fixed at $D$,
and under the constraint that $x \le R \le D/\sqrt3$
in~\eqref{eq:constraints-new},
the product $xy$ is minimized when
$x = \frac1{\sqrt3} D$ and $y = \left(1 - \frac1{\sqrt3}\right) D$.
Thus we have
\begin{equation}\label{eq:case1}
f(x, y) \le \frac23 \cdot
        \frac{\left(\frac1{\sqrt3}\right)^3 + \left(1 - \frac1{\sqrt3}\right)^3}
        {\left(\frac1{\sqrt3}\right)^2 + \left(1 - \frac1{\sqrt3}\right)^2} D
= \frac{2(4 - \sqrt3)}{13} D
= 0.3489\ldots D.
\end{equation}

\paragraph{Case 2.}

We next consider the case\footnote{This case, when $x+y <D$, 
has been mistakenly overlooked in the proof given in~\cite{DT09}.}
that $y = y_1$.
With $y$ fixed,
the function $f(x,y)$ is maximized when $x$ is as large as possible because
\begin{align*}
\frac{\partial}{\partial x} f(x, y) &= \frac23 \cdot
        \left( \frac{3x^2}{x^2 + y^2} - \frac{x^3 + y^3}{(x^2 + y^2)^2}2x \right)
\\
&= \frac23 \cdot \frac{3x^2(x^2 + y^2) - (x^3 + y^3)2x}{(x^2 + y^2)^2}
\\
&= \frac23 \cdot \frac{x(x^3 + 3xy^2 - 2y^3)}{(x^2 + y^2)^2}
\\
&\ge \frac23 \cdot \frac{x(y^3 + 3y^3 - 2y^3)}{(x^2 + y^2)^2} \ge 0.
\end{align*}
Thus for $y= \sqrt{R^2 - D^2/4}$ and 
under the constraint that $x \le R$ in~\eqref{eq:constraints-new},
the function $f(x, y)$ is maximized when
$x = R$ and
$y = \sqrt{R^2 - D^2/4} = \sqrt{x^2 - D^2/4}$.
It follows that
$$
\frac{\dd x}{\dd R} = 1
\qquad\textup{and}\qquad
\frac{\dd y}{\dd R} = \frac{\dd \sqrt{x^2 - D^2/4}}{\dd R}
= \frac{x}{\sqrt{x^2 - D^2/4}} = x/y.
$$

Let $h_2(R) = f(R, \sqrt{R^2 - D^2/4})$.
We next show that $h_2(R)$ is increasing in $R$.
Taking the derivative, we have
\begin{align*}
\frac{\dd}{\dd R} h_2(R) &= \frac23 \cdot
        \left( \frac{3x^2 \frac{\dd x}{\dd R} + 3y^2 \frac{\dd y}{\dd R}}{x^2 + y^2}
                - \frac{x^3 + y^3}{(x^2 + y^2)^2}
                        \left( 2x \frac{\dd x}{\dd R} + 2y \frac{\dd y}{\dd R} \right)
        \right)
\\
&= \frac23 \cdot
        \left( \frac{3x^2 + 3y^2 (x/y)}{x^2 + y^2}
                - \frac{x^3 + y^3}{(x^2 + y^2)^2}(2x + 2y(x/y))
        \right)
\\
&= \frac23 \cdot
        \frac{(3x^2 + 3xy)(x^2 + y^2) - (x^3 + y^3)(2x + 2x)}{(x^2 + y^2)^2}
\\
&= \frac23 \cdot
        \frac{(3x^4 + 3x^2y^2 + 3x^3y + 3xy^3) - (4x^4 + 4xy^3)}{(x^2 + y^2)^2}
\\
&= \frac23 \cdot
        \frac{(x^4 + 3x^2y^2 + 3x^3y + xy^3) - (2x^4 + 2xy^3)}{(x^2 + y^2)^2}
\\
&= \frac23 \cdot
        \frac{x^4}{(x^2 + y^2)^2} \cdot \big( (1+y/x)^3 - 2 - 2 (y/x)^3 \big).
\end{align*}
Substituting $z = y/x$,
we simplify the last factor
$(1+y/x)^3 - 2 - 2 (y/x)^3$
in the resulting expression above to
$$
h_3(z) = (1+z)^3 - 2 - 2 z^3.
$$
To show that $\frac{\dd}{\dd R} h_2(R) > 0$,
it remains to show that $h_3(z) > 0$.
For $0 \le z \le 1$,
the function $h_3(z)$ is increasing in $z$ because
$$
\frac{\dd}{\dd z} h_3(z) = 3(1+z)^2 - 6 z^2  = -3(1-z)^2 + 6 
\geq 6-3 > 0.
$$
Recall that $x \ge y$.
If $R \le \frac{3(4 - \sqrt3)}{13} D$,
then we would easily have
$$
f(x, y) = \frac23 \cdot \frac{x^3 + y^3}{x^2 + y^2}
\le \frac23 \cdot \frac{x^3}{x^2}
= \frac23 x \le \frac 23 R \le \frac{2(4 - \sqrt3)}{13} D,
$$
which matches the upper bound in case~1.
Now suppose that $R > \frac{3(4 - \sqrt3)}{13} D$.
Then
$$ D/R < \frac{13}{3(4 - \sqrt3)} 
{\rm \ and \ }
z = y/x = \sqrt{1 - (D/R)^2/4} > 
\sqrt{1 - \left(\frac{13}{3(4 - \sqrt3)}\right)^2 \bigg{/} 4 }
= 0.2955\ldots. $$
It follows that
$$ h_3(z) > h_3\left(\sqrt{1 - \left(\frac{13}{3(4 - \sqrt3)}\right)^2 \bigg{/} 4 }\right)
= 0.1226\ldots > 0, $$
hence 
$$ \frac{\dd}{\dd R} h_2(R) > 0. $$
We have shown that the function $h_2(R)$ is increasing in $R$.
Then,
under the constraint that $R \le D / \sqrt3$ in~\eqref{eq:constraints-new},
$h_2(R)$ is maximized when $R = \frac1{\sqrt3} D$.
Correspondingly,
$f(x,y)$ is maximized when
$x = \frac1{\sqrt3} D$ and $y = \frac1{\sqrt{12}} D$.
Thus
\begin{equation}\label{eq:case2}
f(x, y) \le \frac23 \cdot
        \frac{\left(\frac1{\sqrt3}\right)^3 + \left(\frac1{\sqrt{12}}\right)^3}
        {\left(\frac1{\sqrt3}\right)^2 + \left(\frac1{\sqrt{12}}\right)^2} D
= \frac{\sqrt3}5 D
= 0.3464\ldots D, 
\end{equation}
which is (slightly) smaller than the upper bound obtained in case~1.
This proves the upper bound in \eqref{E10}.

\paragraph {Centrally symmetric body.}
Assume now that $Q$ is centrally symmetric with respect to a point $q$.
We repeat the same ``double sector'' argument. It is enough to observe
that: (i) the center of $\Omega$ coincides with $q$, that is, $o=q$; and (ii) $x=y
\leq \frac{1}{2} \cdot D$ for any double sector $W$.
By \eqref{E4}, the average distance calculation yields now
$$ \mu_{(W \cap Q)}(o) \leq \frac{2x^3}{3(1-\eps)^2 \cdot x^2}
= \frac{2x}{3(1-\eps)^2} \leq \frac{D}{3(1-\eps)^2}, $$
and by taking the limit when $\eps$ tends to zero, we obtain
$$  \mu_Q(o) \leq \frac{D}{3}, $$
as required. The proof of Theorem~\ref{T2} is now complete.
\endproof

\section{Applications} \label{sec:app}

\noindent {\bf 1.}
Carmi, Har-Peled and Katz~\cite{CHK05} showed that given a convex
polygon $Q$ with $n$ vertices, and a parameter $\eps>0$, one can
compute an $\eps$-approximate Fermat-Weber center $q \in Q$ in
$O(n+1/\eps^4)$ time such that $\mu_Q(q) \leq (1+\eps) \mu^*_Q$.
Abu-Affash and Katz~\cite{AK08} gave a simple $O(n)$-time algorithm for
computing the center $q$ of the smallest disk enclosing $Q$, and showed that
$q$ approximates the Fermat-Weber center of $Q$, with
$\mu_Q(q) \leq \frac{25}{6 \sqrt3} \mu^*_Q$.
Our Theorems~\ref{T1} and~\ref{T2}, combined with their analysis,
improves the approximation ratio to about $2.09$:
$$ \mu_Q(q) \leq \frac{12(4-\sqrt3)}{13} \mu^*_Q. $$

\smallskip
\noindent {\bf 2.}
The value of the constant $c_1$ (i.e., the infimum of $\mu^*_Q/\Delta(Q)$
over all convex bodies $Q$ in the plane) plays a key role in the following
load balancing problem introduced by Aronov, Carmi and Katz~\cite{ACK06}.
We are given a convex body $D$ and $m$ points $p_1,p_2,\ldots , p_m$
representing {\em facilities} in the interior of $D$. Subdivide $D$ into $m$
convex regions, $R_1,R_2,\ldots ,R_m$, of equal area such that $\sum_{i=1}^m
\mu_{p_i}(R_i)$ is minimal. Here $\mu_{p_i}(R_i)$ is the {\em cost} associated
with facility $p_i$, which may be interpreted as the average travel time from the
facility to any location in its designated region, each of which has the same area.
One of the main results in~\cite{ACK06} is a $(8+\sqrt{2\pi})$-factor approximation
in the case that $D$ is an $n_1\times n_2$ rectangle for some integers $n_1,n_2\in \NN$.
This basic approximation bound is then used for several other cases, e.g.,
subdividing a convex fat domain $D$ into $m$ convex regions $R_i$.

By substituting $c_1=\frac{1}{6}$ (Theorem~\ref{T1}) into the analysis
in~\cite{ACK06}, the upper bound for the approximation ratio improves from
$8+\sqrt{2\pi}\approx 10.5067$ to $7+\sqrt{2\pi}\approx 9.5067$.
It can be further improved by optimizing another parameter
used in their calculation. Let $S$ be a unit square and let $s \in S$
be an arbitrary point in the square. Aronov et al.~\cite{ACK06} used the
upper bound $\mu_S(s)\leq \frac{2}{3}\sqrt{2}\approx 0.9429$. It is
clear that $\max_{s\in S}\mu_S(s)$ is attained if $s$ is a vertex of $S$.
The average distance of $S$ from such a vertex, say $v$, is $\mu_S(v)
= \frac{1}{3}\left(\sqrt{2}+\ln(1+\sqrt{2})\right)\approx 0.7652,$
and so $\mu_S(s)\leq \frac{1}{3}\left(\sqrt{2}+\ln(1+\sqrt{2})\right)$,
for any $s\in S$. With these improvements, the upper
bound on the approximation ratio becomes
$7+\frac{\sqrt{\pi}}{2}\left(\sqrt{2}+\ln(1+\sqrt{2})\right)\approx 9.0344$.

\paragraph{Acknowledgment.} We are grateful to Alex Rand for
stimulating discussions.

\vspace{-.6\baselineskip}

\end{document}